\documentclass[11pt, notitlepage]{article}   

\usepackage[title]{appendix}

\usepackage{amsmath,amsthm,amsfonts,hyperref}   

\usepackage{bbm} 
\usepackage{amscd, array}
\usepackage{amsfonts}
\usepackage{amssymb}
\usepackage{color}      
\usepackage{epsfig}
\usepackage{graphicx}           
\usepackage{algorithm}
\usepackage[noend]{algpseudocode}

\usepackage{tikz}
\usetikzlibrary{shapes}
\usetikzlibrary{matrix}
\usetikzlibrary{positioning}
\usetikzlibrary{fit}

\tikzstyle{res}=[circle,thick,minimum size=4mm,draw=black,fill=red,inner sep=1pt]
\tikzstyle{non-res}=[circle,thick,minimum size=4mm,draw=black,inner sep=1pt]
\tikzstyle{light-res}=[circle,thick,minimum size=4mm,draw=black,fill=red!40,inner sep=1pt]
\tikzstyle{blue}=[circle,thick,minimum size=4mm,draw=black,fill=blue!20,inner sep=1pt]

\newtheorem{theorem}{Theorem}[section]

\def\argmin{{\hbox{argmin}}}

\theoremstyle{remark}
\newtheorem*{remark}{Remark}
\theoremstyle{theorem}

\newtheorem{cor}[theorem]{Corollary}

\newtheorem{definition}[theorem]{Definition}
\newtheorem{lem}[theorem]{Lemma}
\newtheorem{thm}[theorem]{Theorem}

\newcommand{\ftdim}{\textnormal{ftdim}}
\newcommand{\dist}{\textnormal{dist}}
\newcommand{\edim}{\textnormal{edim}}
\newcommand{\ftedim}{\textnormal{ftedim}}

\newcommand{\ldim}{\textnormal{ldim}}
\newcommand{\ftldim}{\textnormal{ftldim}}
\newcommand{\xdim}{\textnormal{xdim}}
\newcommand{\ftxdim}{\textnormal{ftxdim}}

\newcommand{\adim}{\textnormal{adim}}
\newcommand{\ftadim}{\textnormal{ftadim}}
\newcommand{\mdim}{\textnormal{mdim}}

\newcommand{\ek}{\textnormal{ek}}
\newcommand{\mc}{\textnormal{mc}}

\def\finf{\mathop{{\rm I}\kern -.27 em {\rm F}}\nolimits}

\usepackage[colorinlistoftodos,textsize=tiny]{todonotes}
\newcommand{\Comments}{1}
\newcommand{\mynote}[2]{\ifnum\Comments=1\textcolor{#1}{#2}\fi}
\newcommand{\mytodo}[2]{\ifnum\Comments=1%
  \todo[linecolor=#1!80!black,backgroundcolor=#1,bordercolor=#1!80!black]{#2}\fi}

\begin{document}

\title{Fault tolerance for metric dimension and its variants}
\author{Jesse Geneson and Shen-Fu Tsai}

\maketitle

\begin{abstract}
    Hernando et al. (2008) introduced the fault-tolerant metric dimension $\ftdim(G)$, which is the size of the smallest resolving set $S$ of a graph $G$ such that $S-\left\{s\right\}$ is also a resolving set of $G$ for every $s \in S$. They found an upper bound $\ftdim(G) \le \dim(G) (1+2 \cdot 5^{\dim(G)-1})$, where $\dim(G)$ denotes the standard metric dimension of $G$. It was unknown whether there exists a family of graphs where $\ftdim(G)$ grows exponentially in terms of $\dim(G)$, until recently when Knor et al. (2024) found a family with $\ftdim(G) = \dim(G)+2^{\dim(G)-1}$ for any possible value of $\dim(G)$. We improve the upper bound on fault-tolerant metric dimension by showing that $\ftdim(G) \le \dim(G)(1+3^{\dim(G)-1})$ for every connected graph $G$. Moreover, we find an infinite family of connected graphs $J_k$ such that $\dim(J_k) = k$ and $\ftdim(J_k) \ge 3^{k-1}-k-1$ for each positive integer $k$. Together, our results show that \[\lim_{k \rightarrow \infty} \left( \max_{G: \text{ } \dim(G) = k} \frac{\log_3(\ftdim(G))}{k} \right) = 1.\] In addition, we consider the fault-tolerant edge metric dimension $\ftedim(G)$ and bound it with respect to the edge metric dimension $\edim(G)$, showing that \[\lim_{k \rightarrow \infty} \left( \max_{G: \text{ } \edim(G) = k} \frac{\log_2(\ftedim(G))}{k} \right) = 1.\] We also obtain sharp extremal bounds on fault-tolerance for adjacency dimension and $k$-truncated metric dimension. Furthermore, we obtain sharp bounds for some other extremal problems about metric dimension and its variants. In particular, we prove an equivalence between an extremal problem about edge metric dimension and an open problem of Erd\H{o}s and Kleitman (1974) in extremal set theory.
\end{abstract}

\section{Introduction}
Given an ordered set of vertices $S = \left\{v_1, v_2, \dots, v_k\right\}$ in a graph $G$, define the \textit{distance vector} of $u \in V(G)$ with respect to $S$ as \[d_S(u) = (\dist(u,v_1),\dist(u,v_2),\dots,\dist(u,v_k)).\] We say that $S$ is a \textit{resolving set} for a graph $G$ if $d_S(u)$ is unique for every vertex $u \in V(G)$. We say that such a set $S$ \textit{resolves} the vertices of $G$. Another name for a resolving set is a \textit{metric basis}. The \textit{metric dimension} $\dim(G)$ of $G$ is defined \cite{slater, harary} as the minimum size of any resolving set (i.e. metric basis) for $G$. Most papers about metric dimension restrict $G$ to be a connected graph. Here, we allow $G$ to be any simple undirected graph, and we say that $\dist(u,v) = \infty$ if $u$ and $v$ are in different connected components.

One of the original applications of metric dimension was for navigation in a graph-structured framework \cite{khuller}. In particular, suppose that a robot is navigating some terrain which can be modeled with a graph. Certain vertices in the graph have landmarks, and the robot has a sensor which can measure its distance in the graph to every landmark. The goal of metric dimension is to find the minimum number of landmarks so that the robot can always determine its location, for any possible vertex that it could visit. 

Hernando et al \cite{ftdim08} introduced a variant of metric dimension in which the landmarks are allowed to fail. The fault-tolerant metric dimension $\ftdim(G)$ of $G$ is the minimum size of a nonempty set of vertices $S$ such that $S-\left\{s\right\}$ resolves the vertices of $G$ for all $s \in S$, and we call such an $S$ a \textit{fault-tolerant resolving set}. Observe that the definition must require $S$ to be nonempty, or else we would have $\ftdim(G) = 0$ for all graphs $G$, since the definition of fault-tolerant resolving set is vacuously true for the empty set. A number of papers \cite{beardon, aem, eyr, ery, ftdim24} have investigated a generalization of this variant which allows larger faults. In particular, the $k$-metric dimension $\mdim_k(G)$ of $G$ is defined as the minimum size of a set of vertices $S$ such that $S-T$ resolves $G$ for all $T \subseteq S$ with $|T| = k-1$. Alternatively, we say that $S$ is a $k$-metric basis for $G$ if every pair of vertices in $G$ is distinguished by at least $k$ vertices of $S$. We can define $\mdim_k(G)$ as the minimum size of a $k$-metric basis for $G$. Note that $\dim(G) = \mdim_1(G)$ and $\ftdim(G) = \mdim_2(G)$.

\begin{remark}
    For all graphs $G$, we have $\ftdim(G) \ge \dim(G)+1$. Indeed, suppose that $S$ is a fault-tolerant resolving set for $G$ with $|S| = \ftdim(G)$. Then, for all $s \in S$, the set $S-\{s\}$ is a resolving set for $G$, so $\dim(G) \le |S|-1 = \ftdim(G)-1$.
\end{remark}

It is easy to find graphs that attain this bound, e.g., $G = K_n$. As for bounds in the other direction, Hernando et al \cite{ftdim08} proved that $\ftdim(G) \le \dim(G) (1+2 \cdot 5^{\dim(G)-1})$ for all graphs $G$. In general we have $\ftdim(G) \ge \dim(G)+1$, but it was not known whether there exists a family of graphs where $\ftdim(G)$ grows exponentially with respect to $\dim(G)$, until Knor et al \cite{ftdim24} proved that there exists a family of connected graphs $G$ with $\ftdim(G) = \dim(G)+2^{\dim(G)-1}$, for any possible value of $\dim(G)$. Given the results of Hernando et al and Knor et al, a remaining problem is to determine how large $\ftdim(G)$ can grow with respect to $\dim(G)$.

For each positive integer $t$, let $m(t)$ be the maximum possible value of $\ftdim(G)$ over all nonempty graphs $G$ with $\dim(G) = t$. The results of Hernando et al and Knor et al show that \[t+2^{t-1} \le m(t) \le t(1+2 \cdot 5^{t-1}).\] Given that the known exponential upper and lower bounds have different bases ($5$ and $2$ respectively), we investigated the problem of determining the true base of the exponent. In particular, we focused on finding the limit \[\lim_{t \rightarrow \infty} \frac{\log(m(t))}{t}.\] In this paper, we show that the true base of the exponent is $3$, i.e., \[\lim_{t \rightarrow \infty} \frac{\log_3(m(t))}{t} = 1.\] We also consider fault-tolerance for some other variants of metric dimension, and we investigate the analogous problem of determining how large the fault-tolerant versions can grow with respect to their corresponding variants. 

For the upper bounds in this paper, we create a fault-tolerant resolving set from an arbitrary resolving set for each variant of metric dimension. Our general strategy for creating the fault-tolerant resolving set is taking the closed neighborhood of the resolving set, and then adding a small number of vertices to handle any remaining unresolved pairs. For the lower bounds in this paper, we construct graphs for which the smallest possible fault-tolerant resolving set is nearly the same size as the closed neighborhood in the upper bound.

For each variant that we consider, we show that the base of the exponent for fault-tolerance is the same as the base of the exponent for maximum degree. In the case of standard metric dimension, it is known \cite{gkl} that the maximum possible degree of a vertex in a graph of metric dimension $k$ is $3^k - 1$, and the maximum possible degree of a vertex in a resolving set of size $k$ is at most $3^{k-1}$. 

Many variants of metric dimension have been investigated in the literature. For example, Kelenc et al \cite{kelenc} introduced a variant called the edge metric dimension. For an edge $e = \left\{u,v\right\} \in E(G)$, define $\dist(e, w) = \min(\dist(w,u),\dist(w,v))$. Given an ordered set of vertices $S = \left\{v_1, v_2, \dots, v_k\right\}$ in a graph $G$, we define the \textit{distance vector} of $e \in E(G)$ with respect to $S$ as \[d_S(u) = (\dist(e,v_1),\dist(e,v_2),\dots,\dist(e,v_k)).\] We say that $S$ is an \textit{edge resolving set} for a graph $G$ if $d_S(e)$ is unique for every edge $e \in E(G)$. We say that such a set $S$ \textit{resolves} the edges of $G$. Another name for an edge resolving set is an \textit{edge metric basis}. The \textit{edge metric dimension} $\edim(G)$ of $G$ is defined as the minimum size of any edge resolving set (i.e. edge metric basis) for $G$. It was shown in \cite{gmd} that the maximum possible degree of a vertex in a graph $G$ with $\edim(G) = k$ is $2^k$. We show that the maximum possible degree of a vertex in an edge resolving set of size $k$ is $2^{k-1}$.

As with standard metric dimension, the fault-tolerant edge metric dimension $\ftedim(G)$ of $G$ is defined \cite{ftedim21} to be the minimum size of a nonempty set of vertices $S$ such that $S-\left\{s\right\}$ resolves the edges of $G$ for all $s \in S$. We prove that \[\lim_{k \rightarrow \infty} \left( \max_{G: \text{ } \edim(G) = k} \frac{\log_2(\ftedim(G))}{k} \right) = 1.\]

Jannesari and Omoomi \cite{jannesari} introduced the adjacency dimension. Given an ordered set of vertices $S = \left\{v_1, v_2, \dots, v_k\right\}$ in a graph $G$, we define the \textit{adjacency vector} of $u \in V(G)$ with respect to $S$ as \[a_S(u) = (\min(2,\dist(u,v_1)),\min(2,\dist(u,v_2)),\dots,\min(2,\dist(u,v_k))).\] We say that $S$ is an \textit{adjacency resolving set} for a graph $G$ if $a_S(u)$ is unique for every vertex $u \in V(G)$. We say that such a set $S$ \textit{$1$-resolves} the vertices of $G$. Another name for an adjacency resolving set is an \textit{adjacency basis}. 

The \textit{adjacency dimension} $\adim(G)$ of $G$ is defined as the minimum size of any adjacency resolving set (i.e. adjacency basis) for $G$. If we extend the robot analogy, then the goal of the adjacency dimension is to find the minimum number of landmarks so that the robot can always determine its location, assuming that its sensors can only detect landmarks on vertices that are within a distance of $1$. It was shown in \cite{gy} that the maximum possible degree of a vertex in a graph $G$ with $\adim(G) = k$ is $2^k+k-1$. We show that the maximum possible degree of a vertex in an adjacency resolving set of size $k$ is $2^{k-1}+k-1$. The adjacency dimension has also been extended to a variant called the $k$-truncated metric dimension \cite{beardon, aem, eyr, fr, frongillo}, where the sensors can detect landmarks on vertices that are within a distance of $k$. 

Given an ordered set of vertices $S = \left\{v_1, v_2, \dots, v_j\right\}$ in a graph $G$, we define the \textit{$k$-truncated vector} of $u \in V(G)$ with respect to $S$ as \[d_{S,k}(u) = (\min(k+1,\dist(u,v_1)),\min(k+1,\dist(u,v_2)),\dots,\min(k+1,\dist(u,v_k))).\] We say that $S$ is a \textit{$k$-truncated resolving set} for a graph $G$ if $d_{S,k}(u)$ is unique for every vertex $u \in V(G)$. We say that such a set $S$ \textit{$k$-resolves} the vertices of $G$. Another name for a $k$-truncated resolving set is an \textit{$k$-truncated metric basis}. The \textit{$k$-truncated metric dimension} $\dim_k(G)$ of $G$ is defined as the minimum size of any $k$-truncated resolving set (i.e. $k$-truncated metric basis) for $G$. Note that for all graphs $G$, we have $\dim_1(G) = \adim(G)$ and $\dim_k(G) = \dim(G)$ for all $k$ that are sufficiently large with respect to $G$. In general, we have $\dim_k(G) \ge \dim(G)$ for all graphs $G$. It was shown in \cite{frongillo} that the maximum possible degree of a vertex in a graph $G$ with $\dim_k(G) = j$ is $3^j - 1$ for all $k \ge 2$, and $2^j+j-1$ in the case that $k = 1$. We show that the maximum possible degree of a vertex in a $k$-resolving set of size $j$ is $3^{j-1}$.

As with standard metric dimension, we define the fault-tolerant adjacency dimension $\ftadim(G)$ of $G$ to be the minimum size of a nonempty set of vertices $S$ such that $S-\left\{s\right\}$ $1$-resolves the vertices of $G$ for all $s \in S$. Similarly, we define the fault-tolerant $k$-truncated metric dimension $\ftdim_k(G)$ of $G$ to be the minimum size of a nonempty set of vertices $S$ such that $S-\left\{s\right\}$ $k$-resolves the vertices of $G$ for all $s \in S$. As in the case without fault-tolerance, we have $\ftdim_1(G) = \ftadim(G)$ and $\ftdim_k(G) = \ftdim(G)$ for all $k$ that are sufficiently large with respect to $G$. Moreover, we have $\ftdim_k(G) \ge \ftdim(G)$ for all graphs $G$. In this paper, we prove for all $k \ge 2$ that \[\lim_{j \rightarrow \infty} \left( \max_{G: \text{ } \dim_k(G) = j} \frac{\log_3(\ftdim_k(G))}{j} \right) = 1.\] However, for the case when $k = 1$, we show that \[\lim_{j \rightarrow \infty} \left( \max_{G: \text{ } \dim_1(G) = j} \frac{\log_2(\ftdim_1(G))}{j} \right) = 1.\]

There is a variant of metric dimension called local metric dimension, denoted $\ldim(G)$, which was introduced by Okamoto et al \cite{okamoto}.  It is like standard metric dimension, except we only have to distinguish pairs of adjacent vertices instead of all pairs of vertices \cite{abrishami, br, ghalavand, ghalavand0}. In particular, we say that $S$ is a \textit{local resolving set} for a graph $G$ if $d_S(u) \neq d_S(v)$ for all $\left\{u,v\right\} \in E(G)$. We say that such a set $S$ \textit{locally resolves} the vertices of $G$. The \textit{local metric dimension} $\ldim(G)$ is the minimum size of any local resolving set for $G$. For example, $\ldim(K_{1,n}) = 1$ for all $n$, since any vertex is a local resolving set for a star graph. 

We define $\ftldim(G)$ analogously to the other variants. Specifically, the fault-tolerant local metric dimension $\ftldim(G)$ of $G$ is the minimum size of a nonempty set of vertices $S$ such that $S-\left\{s\right\}$ locally resolves the vertices of $G$ for all $s \in S$. We prove some characterization results about fault tolerance for local metric dimension.

Consider an arbitrary variant $\xdim(G)$ of metric dimension, and suppose that $\xdim(G)$ is defined to be the minimum size of an $\xdim$-resolving set for $G$. We can define $\ftxdim(G)$ to be the minimum size of a nonempty set of vertices $S$ such that $S-\left\{s\right\}$ is an $\xdim$-resolving set for all $s \in S$. For each variant $\xdim(G)$ of metric dimension, we have the following bound which we use several times in this paper.

\begin{lem}\label{lem:ftxdim}
    For all graphs $G$ and for each variant $\xdim(G)$ of metric dimension, we have $ftxdim(G) \ge \xdim(G)+1$.
\end{lem}

\begin{proof}
    Suppose that $S$ is a fault-tolerant $\xdim$-resolving set for $G$ with $|S| = \ftxdim(G)$. Then, for all $s \in S$, the set $S-\{s\}$ is an $\xdim$-resolving set for $G$, so $\xdim(G) \le |S|-1 = \ftxdim(G)-1$.
\end{proof}

In addition to the extremal bounds in this paper, we also characterize the graphs $G$ for which $\ftdim(G) = 2$, $\ftdim(G) = n$, $\ftedim(G) = 2$, $\ftedim(G) = n$, $\ftadim(G) = 2$, $\ftadim(G) = n$, $\ftdim_k(G) = 2$, $\ftdim_k(G) = n$, $\ftldim(G) = 2$, and $\ftldim(G) = n$. For each variant $\xdim(G)$ of metric dimension that we consider in this paper, we show that for all graphs $G$ we have $\ftxdim(G) = 2$ if and only if $\xdim(G) = 1$.



In order to prove our bounds for fault tolerance, we also derive sharp bounds for the maximum possible degree of vertices in resolving sets of size $k$, edge resolving sets of size $k$, and $k$-resolving sets of size $j$. We also generalize these bounds for resolving sets by determining the maximum possible number of vertices within distance $j$ of a vertex in a resolving set of size $k$.

In addition to the extremal results about maximum degree, we also consider the maximum possible clique number of graphs with edge metric dimension at most $k$. The best known bounds for this quantity are $O(2^{k/2})$ and $\Omega((\frac{8}{3})^{k/6})$. We show that this problem is equivalent to an open problem of Erd\H{o}s and Kleitman \cite{ek74}, the maximum possible size of a family of subsets of $\left\{1,2,\dots, k\right\}$ such that all pairwise unions are distinct. 

We prove the results on fault tolerance for standard metric dimension in Section~\ref{sec:ftdim}, including extremal bounds and characterizations. In Section~\ref{sec:edim}, we prove the results on fault tolerance for edge metric dimension. We obtain the results on fault tolerance for adjacency dimension in Section~\ref{sec:adim}. We generalize the fault tolerance results for adjacency dimension to truncated metric dimension in Section~\ref{sec:ftdimk}. We discuss fault tolerance for local metric dimension in Section~\ref{sec:ftldim}.

In Section~\ref{sec:dim_extremal}, we determine the maximum possible number of vertices within distance $j$ of a vertex in a resolving set of size $k$, as well as the maximum possible number of vertices within distance $j$ of any vertex in a graph of metric dimension $k$. We prove the equivalence between the maximum possible clique number of a graph of edge metric dimension at most $k$ and the problem of Erd\H{o}s and Kleitman \cite{ek74} in Section~\ref{sec:equiv}. Finally, in Section~\ref{sec:conclusion}, we discuss some related open problems and future research directions.

\section{Fault tolerance for standard metric dimension}\label{sec:ftdim}

If $G$ is a graph of order $1$, then clearly we have $\dim(G) = 0$ and $\ftdim(G) = 1$. If $G$ has order at least $2$, then we have $\dim(G) > 0$ and $\ftdim(G) > 1$. In order to characterize the graphs $G$ with $\ftdim(G) = 2$, we start with a characterization of the graphs $G$ with $\dim(G) = 1$. This characterization is well-known for connected graphs, but we include a proof since this is for all graphs. 

\begin{thm}\label{thm:dim_1}
For all graphs $G$, we have $\dim(G) = 1$ if and only if $G$ has order at least $2$, at most $2$ connected components, at most $1$ connected component of size at least $2$, and all connected components are singletons or paths. 
\end{thm}

\begin{proof}
Suppose that $G$ has order at least $2$, at most $2$ connected components, at most $1$ connected component of size at least $2$, and all connected components of $G$ are singletons or paths. Since $G$ has order at least $2$, we have $\dim(G) > 0$. 

If $G$ has no connected component of size at least $2$, then $G$ must consist of two singletons. In this case, we let $S$ consist of a single vertex of $G$, which is clearly a resolving set. If $G$ has a connected component $C$ of size at least $2$, then $C$ is a path. Let $S$ consist of a single endpoint of $C$. Again, $S$ is clearly a resolving set. 

For the other direction, suppose that $\dim(G) = 1$. Since $\dim(G) > 0$, $G$ must have order at least $2$. There are at most $2$ connected components in $G$, since any vertex $v$ cannot resolve any pair of vertices $x,y$ such that $v,x,y$ are all in different connected components. At most $1$ connected component in $G$ may have size at least $2$, since any vertex $v$ cannot resolve any pair of vertices $x,y$ such that $v$ is in a different connected component from $x$ and $y$. 

Any landmark $v$ in a resolving set of size $1$ for $G$ must have degree at most $1$, or else $v$ would not resolve its neighbors. Any vertex $u$ in $G$ must have degree at most $2$, or else $v$ would not resolve the neighbors of $u$. Moreover, no connected component of $G$ is a cycle, or else $v$ could not resolve the cycle. Thus, every connected component of $G$ is a singleton or a path. 
\end{proof}

As for the maximum possible values of $\dim(G)$, it is easy to see that $\dim(G) \le n-1$ for all graphs $G$ of order $n$, since any set of $n-1$ vertices in $G$ is clearly a resolving set. In the next result, we characterize the graphs $G$ with $\dim(G) = n-1$. This characterization was already found in the case of connected graphs \cite{chartrand}.

\begin{thm}
    For all graphs $G$ of order $n$, we have that $\dim(G) = n-1$ if and only if $G$ is a complete graph or a graph with no edges.
\end{thm}

\begin{proof}
If $G$ is a complete graph or a graph with no edges of order $n$, then $\dim(G) \ge n-1$ since any set of vertices in $G$ of size $n-2$ would not distinguish the two vertices that are not in the set. Thus, $\dim(G) = n-1$ if $G$ is a complete graph or a graph with no edges of order $n$. 

In the other direction, suppose that $G$ is a graph of order $n$ with $\dim(G) = n-1$. The theorem is clearly true for $n = 1$, so we suppose that $n \ge 2$. Let $u,v$ be two arbitrary vertices of $G$. We split the proof into two cases. 

In the first case, suppose that $\{u,v\} \in E(G)$. Then for all $x \in V(G)$, we must have $\{u,x\} \in E(G)$, or else the set $V(G)- \{v,x\}$ would resolve $G$. However, this implies for all $x,y \in V(G)$ that we must have $\{x,y\} \in E(G)$, or else the set $V(G)-\{u,y\}$ would resolve $G$. Thus, $G$ is a complete graph. 

In the second case, suppose that $\{u,v\} \not \in E(G)$. By an analogous argument to the first case, $G$ must have no edges.
\end{proof} 

Now, we turn to fault tolerance for standard metric dimension. In the next result, we characterize the graphs $G$ of order at least $2$ which attain the minimum possible value of $\ftdim(G)$.

\begin{thm}
For any graph $G$, $\ftdim(G) = 2$ if and only if $G$ has order at least $2$, at most $2$ connected components, at most $1$ connected component of size at least $2$, and all connected components are singletons or paths.
\end{thm}

\begin{proof}
Suppose that $G$ has order at least $2$, at most $2$ connected components, at most $1$ connected component of size at least $2$, and all connected components of $G$ are singletons or paths. Since $G$ has order at least $2$, we have $\dim(G) > 0$, so $\ftdim(G) > 1$. 

If $G$ has no connected component of size at least $2$, then $G$ must consist of two singletons. In this case, we let $S = V(G)$, which is clearly a fault-tolerant resolving set. If $G$ has a connected component $C$ of size at least $2$, then $C$ is a path. Let $S$ consist of the endpoints of $C$. For all $s \in S$, the set $S’ = S-\{s\}$ consists of a single endpoint of C. Since $S’$ resolves $G$, we have that $S$ is a fault-tolerant resolving set. 

For the other direction, suppose that $\ftdim(G) = 2$. Then, $G$ must have order at least $2$, and by Lemma~\ref{lem:ftxdim} we must have $\dim(G) \le \ftdim(G)-1 = 1$. Thus, by Theorem~\ref{thm:dim_1}, $G$ has at most $2$ connected components, at most $1$ connected component of size at least $2$, and all connected components of $G$ are singletons or paths.
\end{proof}

\begin{cor}
    For all graphs $G$, we have $\ftdim(G) = 2$ if and only if $\dim(G) = 1$.
\end{cor}

It is easy to see that $\ftdim(G) \le n$ for all graphs $G$ of order $n$, since any set of $n-1$ vertices in $G$ is a resolving set. The next result shows that this upper bound is sharp.

\begin{thm}\label{thm:ftdim<n}
    For all graphs $G$ of order $n$, we have $\ftdim(G) < n$ if and only if there exists $v \in V(G)$ such that for all $u \in V(G)$ with $u \neq v$ we have $N(u)-\left\{v\right\} \neq N(v)-\left\{u\right\}$.
\end{thm}

\begin{proof}
Suppose that there exists a vertex $v \in V(G)$ such that for all $u \in V(G)$ with $u \neq v$ we have $N(u)-\left\{v\right\} \neq N(v)-\left\{u\right\}$. Let $S = V(G)-\left\{v\right\}$. Then for all $s \in S$, the set $S-\left\{s\right\}$ distinguishes $v$ from $s$, since $N(v)-\left\{s\right\} \neq N(s)-\left\{v\right\}$. Thus, $\ftdim(G) < n$.

For the other direction, suppose that $\ftdim(G) < n$. Then there exists a vertex $v \in V(G)$ such that $V(G)-\left\{v\right\}$ is a fault-tolerant resolving set. Then for every vertex $u \in V(G)-\left\{v\right\}$, the set $V(G)-\left\{u,v\right\}$ is a resolving set. In particular, $V(G)-\left\{u,v\right\}$ resolves $u$ and $v$. So, there must exist some vertex $w \in V(G)-\left\{u,v\right\}$ such that $w$ is adjacent to $u$ or $v$ but not both. Thus, $N(u)-\left\{v\right\} \neq N(v)-\left\{u\right\}$.
\end{proof}

\begin{cor}\label{cor:ftdim_n}
    For all graphs $G$ of order $n$, we have $\ftdim(G) = n$ if and only if for all $v \in V(G)$ there exists $u \in V(G)$ with $u \neq v$ such that $N(u)-\left\{v\right\} = N(v)-\left\{u\right\}$.
\end{cor}

By Corollary~\ref{cor:ftdim_n}, it is clear that we have $\ftdim(K_n) = n$ for all positive integers $n$. If $G$ is a graph of order $n$ with no edges, then we also have $\ftdim(G) = n$ by the same corollary. There are many other graphs $G$ of order $n$ with $\ftdim(G) = n$, e.g., complete bipartite graphs of order $n$ with both parts of size at least $2$.

Now, we turn to extremal bounds on $\ftdim(G)$ with respect to $\dim(G)$. We start with an upper bound, which we will show is sharp up to the base of the exponent.

\begin{thm}
    For every graph $G$ of order greater than $1$, we have $\ftdim(G) \le \dim(G)(2+3^{\dim(G)-1})$.
\end{thm}

\begin{proof}
Given a resolving set $S$ for $G$, define \[S' = S \cup \bigcup_{s \in S} N(s).\] For each $s \in S$, let $T_s = S' - \left\{s\right\}$.
    We claim that $T_s$ resolves all pairs of vertices in $G$ except possibly $s$ and exactly one other vertex $u$. First, we show that if vertices $u,v\ne s$ then $T_s$ resolves $u,v$. Suppose that $T_s$ does not resolve them. Given $S$ resolves them, without loss of generality assume that $\dist(s,u)<\dist(s,v)$. Clearly $s$ is not isolated. Since for any vertex $x\ne s$, \[\dist(s, x) = 1+\min_{x' \in N(s)} \dist(x', x),\]
    assuming that $\argmin_{y\in N(s)}\dist(y,u)=u'$  and $\argmin_{y\in N(s)}\dist(y,v)=v'$ we have
     \[\dist(u', u) < \dist(v',v)\le\dist(u',v),\]
     so $T_s$ resolves $u,v$. Next, we show that there do not exist distinct vertices $u,v\ne s$ such that $T_s$ does not resolve $s,u$ and $s,v$. If there exist such vertices $u,v$, then $u,v\not \in T_s$ and for any vertex $x\in T_s$ we have $\dist(u,x)=\dist(s,x)=\dist(v,x)$, i.e., $T_s$ does not resolve $u,v$. This contradicts our previous claim.
    
    If $T_s$ does not resolve $s$ from some other vertex $u$, then we define $u_s = u$, and otherwise we define $u_s = s$. Let \[S'' = S' \cup \bigcup_{s \in S} \left\{u_s\right\}.\]
    
    Clearly $S''$ is a resolving set for $G$. Consider an arbitrary vertex $t \in S''$, and let $X = S'' - \left\{t\right\}$. If $t \not \in S$, then $X$ clearly resolves all pairs of vertices in $G$ since $S \subseteq X$. Otherwise, we may assume that $t \in S$. Recall that $T_t$ resolves all pairs of vertices in $G$, except possibly the pair of $t$ and $u_t$. However, $u_t \in X$ and $T_t \subseteq X$, so $X$ resolves all pairs of vertices in $G$. Thus, $S''$ is a fault-tolerant resolving set for $G$.

    For each $s\in S$, we have $|N(s)|\le 3^{|S|-1}$ \cite{gkl} and thus, $|S'| \le |S|(1+3^{|S|-1})$. Therefore, we have \[|S''| \le |S'|+|S| \le |S|(2+3^{|S|-1}),\] completing the proof.
\end{proof}

For the next result, we use the following construction of an infinite family of connected graphs $J_1, J_2, \dots$, which is very similar to a construction from \cite{gmd}. Given a copy of $K_{1, 3^k}$ with center vertex $c$, label each non-center vertex in the copy with a unique ternary string of length $k$ (digits $0$, $1$, $2$). Then, add $2k$ new vertices $s_1, \dots, s_k$ and $r_1, \dots, r_k$ such that $s_i$ and $r_i$ both have an edge to every vertex that is labeled with a ternary string having $0$ in its $i^{\text{th}}$ digit, $r_i$ has an edge to every vertex that is labeled with a ternary string having $1$ in its $i^{\text{th}}$ digit, and $r_i$ has an edge to $s_i$. Let $S = \left\{s_1, \dots, s_k \right\}$. Finally, from the copy of $K_{1,3^k}$ we remove the non-center vertex labeled with the all-$1$s ternary string and $k$ non-center vertices labeled with ternary strings containing exactly $1$ zero and $k-1$ ones. Hence the copy of $K_{1,3^k}$ becomes a copy of $K_{1,3^k-k-1}$.

Note that $\dist(v, s_i)$ is one more than the $i^{\text{th}}$ digit of $v$ for all $i$ with $1 \le i \le k$ and all non-center vertices $v$ in the copy of $K_{1, 3^k-k-1}$. This is because if the $i^{\text{th}}$ digit of $v$ is $0$, then by definition $v$ is adjacent to $s_i$. If the $i^{\text{th}}$ digit of $v$ is $1$, then by definition $v$ is adjacent to $r_i$, which is adjacent to $s_i$. If the $i^{\text{th}}$ digit of $v$ is $2$, then $vcus_i$ is a shortest $v,s_i$-path of length $3$ where the ternary string of $u$ contains $0$ in the $i^{\text{th}}$ digit and $2$ elsewhere. Also note that $s_i$ is the only vertex with distance $0$ to itself. Thus, the distance vectors $d_{v,S}$ are unique for all $v$ that are in $S$ or are non-center vertices in the copy of $K_{1, 3^k-k-1}$.

For all $i \neq j$, we have $d_{r_i,s_j}=2$ because $r_ius_j$ is a shortest $r_i,s_j$-path of length $2$ where $u$ has the all-$0$ ternary string. Thus $d_{r_i, S} \neq d_{r_j,S}$ for all $i\ne j$. Moreover, we have $d_{c,S} \neq d_{r_i,S}$ for all $i$ because $d_{c,s_i}>1=d_{r_i,s_i}$. 
Let $v$ be some non-center vertex in the copy of $K_{1,3^k-k-1}$. We have that $d_{c,S}\ne d_{v,S}$ because $d_{c,S}$ is the all-$2$ vector and the non-center vertex labeled with the all-$1$ ternary string has been removed. Finally, note that $d_{r_i,S}$ has $1$ in the $i^{\text{th}}$ digit and $2$ elsewhere. Thus $d_{r_i,S}\ne d_{v,S}$ because we have removed all non-center vertices labeled with ternary strings containing exactly $1$ zero and $k-1$ ones.


\begin{thm}\label{thm:ftdim_lower}
    For every positive integer $k$, we have $\dim(J_k) = k$ and $\ftdim(J_k) \ge 3^{k-1}-k$.
\end{thm}

\begin{proof}
    The fact that $\dim(J_k) \le k$ follows by definition of $J_k$. Indeed, $S$ is a resolving set for $J_k$ of size $k$. To see that $\dim(J_k) \ge k$, note that the center vertex $c$ has degree $3^k-k-1$, any graph $G$ with $\dim(G) = x$ has maximum degree at most $3^x  - 1$ \cite{gkl}, and $3^k-k-1 > 3^{k-1}-1$ for all $k \ge 1$. Thus, we have $\dim(J_k) = k$.
    
    Given a fault-tolerant resolving set $T$ for $J_k$, note that $T-\left\{s_1\right\}$ must be a resolving set for $J_k$. In $J_k$, consider the set of vertices $N$ which are neighbors of the center vertex $c$. Note that $|S| = 3^k-k-1$ by definition of $J_k$. We partition the vertices of $N$ into $3^{k-1}$ subsets of at most $3$ vertices, where the labels of the vertices in each subset differ only in the first digit of their ternary strings. At most $k+1$ subsets have size less than $3$. Let $R$ be the set of these subsets of size exactly $3$, so $|R| \ge 3^{k-1}-k-1$.
    
    If vertices $v$ and $w$ are in the same subset where the first digit of the label of $v$ is $0$ and the first digit of the label of $w$ is $1$, then the only vertices in $J_k$ that resolve $v$ and $w$ are $v$, $w$, and $s_1$. Indeed, the center vertex $c$ has distance $1$ to both $v$ and $w$. Every non-center vertex from the copy of $K_{1,3^k-k-1}$ that is not equal to $v$ or $w$ has distance $2$ to both $v$ and $w$ through the center vertex. Every vertex $s_i$ with $i \neq 1$ has distance $1$ to both $v$ and $w$ if $v$ and $w$ have $i^{\text{th}}$ digit $0$, distance $2$ to both $v$ and $w$ through vertex $r_i$ if $v$ and $w$ have $i^{\text{th}}$ digit $1$, and distance $3$ to both $v$ and $w$ through a vertex $u$ and the center vertex if $v$ and $w$ have $i^{\text{th}}$ digit $2$, where $u\in N$ is labeled with the all-$0$ ternary string. Every vertex $r_i$ with $i \neq 1$ has distance $1$ to both $v$ and $w$ if $v$ and $w$ have $i^{\text{th}}$ digit $0$ or $1$, and otherwise $r_i$ has distance $3$ to both $v$ and $w$ through the vertex $u$ and the center vertex. Finally, $r_1$ has distance $1$ to both $v$ and $w$. Thus, the only vertices in $J_k$ that resolve $v$ and $w$ are $v$, $w$, and $s_1$.
    
    So, in order for $T-\left\{s_1\right\}$ to be a resolving set for $J_k$, we must have $v \in (T-\left\{u_1\right\})$ or $w \in (T-\left\{u_1\right\})$. Thus, every subset in $R$ must contribute at least one vertex to $T-\left\{s_1\right\}$. Therefore, $T-\left\{s_1\right\}$ must have size at least $3^{k-1}-k-1$.
\end{proof}

\begin{cor} We have
    \[\lim_{k \rightarrow \infty} \left( \max_{G: \text{ } \dim(G) = k} \frac{\log_3(\ftdim(G))}{k} \right) = 1.\]
\end{cor}

If we restrict the maximum in the corollary so that it only ranges over connected graphs $G$ with $\dim(G) = k$, rather than over all graphs $G$ with $\dim(G) = k$, then we still obtain the same limit because the infinite family of graphs $J_1,J_2,\dots$ used in the proof of Theorem~\ref{thm:ftdim_lower} are connected.

\section{Fault tolerance for edge metric dimension}\label{sec:edim}

If a graph $G$ has at most one edge, then it is easy to see that $\edim(G) = 0$ and $\ftedim(G) = 1$. If $G$ has at least two edges, then we must have $\edim(G) \ge 1$ and $\ftedim(G) \ge 2$. In the next result, we characterize the graphs $G$ such that $\edim(G) = 1$. This characterization has already been proved for connected graphs \cite{kelenc}; we prove it here for all graphs. 

\begin{thm}\label{thm:edim_1}
For any graph $G$, $\edim(G)=1$ if and only if $G$ has at least $2$ edges, every component of $G$ is a singleton or a path, at most $2$ components of $G$ have an edge, and at most one component of $G$ has more than one edge.
\end{thm}
\begin{proof}
First, we prove the backward direction. Suppose that $G$ has at least $2$ edges, every component of $G$ is a singleton or a path, at most $2$ components of $G$ have an edge, and at most one component of $G$ has more than one edge. Let $S$ consist of one endpoint vertex for a longest path component of $G$. Then $S$ is an edge-resolving set for $G$.
   
Next, we prove the forward direction. Suppose that $\edim(G)=1$, and let $S$ be an edge resolving set for $G$ with $|S| = 1$. Every landmark has degree at most one, otherwise it does not resolve multiple edges incident to it. All other vertices have degree at most $2$, otherwise no single landmark resolves more than two edges incident to any specific vertex. $G$ cannot have any component that is a cycle, since a cycle needs at least $2$ landmarks to resolve the edges.

If more than two components of $G$ have an edge, then these edges cannot be resolved by a single landmark. If two components of $G$ have more than one edge, then a single landmark in one component cannot resolve the edges in another component. Finally, note that $G$ has at least $2$ edges, since otherwise we would have $\edim(G) = 0$.
\end{proof}

Next, we determine all graphs $G$ that attain the minimum possible value of $\ftedim(G)$ among graphs with at least $2$ edges.

\begin{thm}
For any graph $G$, $\ftedim(G)=2$ if and only if $G$ has at least $2$ edges, every component of $G$ is a singleton or a path, at most $2$ components of $G$ have an edge, and at most one component of $G$ has more than one edge.
\end{thm}
\begin{proof}
First, we prove the backward direction. Suppose that $G$ has at least $2$ edges, every component of $G$ is a singleton or a path, at most $2$ components of $G$ have an edge, and at most one component of $G$ has more than one edge. Let $S$ consist of the two endpoint vertices for a longest path component of $G$. Then $S$ is a fault-tolerant edge-resolving set for $G$.
   
Next, we prove the forward direction. Suppose that $\ftedim(G)=2$. Then, $G$ has at least $2$ edges. Since $\edim(G) \le \ftedim(G)-1 = 1$ by Lemma~\ref{lem:ftxdim}, we can conclude by Theorem~\ref{thm:edim_1} that every component of $G$ is a singleton or a path, at most $2$ components of $G$ have an edge, and at most one component of $G$ has more than one edge.
\end{proof}

\begin{cor}
    For all graphs $G$, we have $\ftedim(G) = 2$ if and only if $\edim(G) = 1$.
\end{cor}

As with standard metric dimension, it is easy to see that $\ftedim(G) \le n$ for all graphs $G$ of order $n$, since any set of $n-1$ vertices in $G$ is an edge resolving set. The next theorem shows that this upper bound is sharp.

\begin{thm}
For all graphs $G$ of order $n$, we have $\ftedim(G)<n$ if and only if there exists $v\in V(G)$ such that for all $u\in V(G)-\{v\}$ that share a neighbor with $v$ we have $N(u)-\{v\}\ne N(v)-\{u\}$.
\end{thm}
\begin{proof}
    Suppose that there exists a vertex $v\in V(G)$ such that for all $u\in V(G)-\{v\}$ that share a neighbor with $v$ we have $N(u)-\{v\}\ne N(v)-\{u\}$. Let $S=V(G)-\{v\}$. Then for all $s\in S$, let $S'=S-\{s\}=V(G)-\{s,v\}$. Consider a pair of edges $e_1$ and $e_2$. The only possibility that $e_1$ and $e_2$ are not resolved by $S'$ is that they share an endpoint vertex and their other endpoint vertices are $v$ and $s$. However, since $v$ and $s$ have a common neighbor, by the assumption we have $N(s)-\{v\}\ne N(v)-\{s\}$. Thus $e_1$ and $e_2$ are resolved by $S'$.

    For the other direction, suppose that $\ftedim(G)<n$. Then there exists a vertex $v\in V(G)$ such that $V(G)-\{v\}$ is a fault-tolerant edge resolving set. Thus for every vertex $u\in V(G)-\{v\}$, the set $V(G)-\{u,v\}$ is an edge resolving set. If $u$ shares a neighbor $y$ with $v$, then we have $N(u)-\{v\}\ne N(v)-\{u\}$. Otherwise, $V(G)-\{u,v\}$ does not edge resolve edges $\{u,y\}$ and $\{v,y\}$.
\end{proof}

\begin{cor}
For all graphs $G$ of order $n$, we have $\ftedim(G)=n$ if and only if for all $v\in V(G)$ there exists $u\in V(G)-\{v\}$ that shares a neighbor with $v$ such that $N(u)-\{v\}=N(v)-\{u\}$.
\end{cor}

For multiple results in this section, we will use the following construction of an infinite family of connected graphs $H_1, H_2, \dots$ from \cite{gmd}. Given a copy of $K_{1, 2^k}$ with center vertex $c$, label each non-center vertex in the copy with a unique binary string of length $k$. Then, add $k$ new vertices $u_1, \dots, u_k$, such that $u_i$ has an edge to the $2^{k-1}$ vertices that are labeled with binary strings having $0$ in the $i^{\text{th}}$ digit. Call the resulting graph $H_k$. It was shown \cite{gmd} that $\edim(H_k) = k$, and in particular that $S = \left\{u_1, \dots, u_k\right\}$ is an edge resolving set for $G$.

It was shown \cite{gmd} that any vertex in a graph $G$ with $\edim(G) \le k$ must have degree at most $2^k$. Here, we find a sharp bound for vertices in an edge resolving set.

\begin{thm}\label{thm:edimdeg}
    Any vertex in an edge resolving set of size $k$ must have degree at most $2^{k-1}$. Moreover, this bound is sharp.
\end{thm}

\begin{proof}
    Let $G = (V,E)$ be a graph, and let $S$ be an edge resolving set for $G$ of size $k$. For each edge $e$ in $G$, let $d_S(e)$ denote the distance vector of $e$ with respect to $S$.
    
    Let $s \in S$. For each edge in $G$ of the form $\left\{s,v\right\}$ for some $v \in V$, the $s$ coordinate of the distance vector will be $0$, and there are at most $2$ possible values for every other coordinate. Thus, the number of edges in $G$ of the form $\left\{s,v\right\}$ is at most $2^{k-1}$.

    For the lower bound, we use the graph $H_k$ from \cite{gmd} which was defined at the beginning of the section. Recall that $\edim(H_k) = k$, $S = \left\{u_1, \dots, u_k\right\}$ is an edge resolving set for $H_k$, and every element of $S$ has degree $2^{k-1}$ \cite{gmd}. Thus, our bound is sharp.
\end{proof}

\begin{cor}
    The maximum possible minimum degree of a graph $G$ with $\edim(G)\le k$ is at most $2^{k-1}$.
\end{cor}

In the next two results, we show that \[\lim_{k \rightarrow \infty} \left( \max_{G: \text{ } \edim(G) = k} \frac{\log_2(\ftedim(G))}{k} \right) = 1.\] If we restrict the maximum so that it only ranges over connected graphs $G$ with $\edim(G) = k$, rather than over all graphs $G$ with $\edim(G) = k$, then we still obtain the same limit because the infinite family of graphs $H_1,H_2,\dots$ used in the proof of Theorem~\ref{thm:ftedim_lower} are connected. We start with the upper bound.

\begin{thm}\label{thm:ftedim_upper}
    For every graph $G$ of order greater than $1$, we have $\ftedim(G) \le \edim(G)(1+2^{\edim(G)})$.
\end{thm}

\begin{proof}
    Given an edge resolving set $S$ for $G$, define \[S' = S \cup \bigcup_{s \in S} N(s).\] For each $s \in S$, let $T_s = S' - \left\{s\right\}$.
    
    We claim that $T_s$ resolves all pairs of edges in $G$ except possibly $\{s,u\}$ and $\{u,v\}$ where $s\in S$. First, we show that if $s \not \in e$ and $s \not \in f$ then $T_s$ resolves $e,f$. Suppose that $T_s$ does not resolve them. Given that $S$ resolves them, without loss of generality assume that $\dist(s,e)<\dist(s,f)$. Clearly $s$ is not isolated. Since for any vertex $x\ne s$, \[\dist(s, x) = 1+\min_{x' \in N(s)} \dist(x', x),\]
    assuming that $\argmin_{y\in N(s)}\dist(y,e)=\eta$  and $\argmin_{y\in N(s)}\dist(y,f)=\phi$ we have
     \[\dist(\eta, e) < \dist(\phi,f)\le\dist(\eta,f),\]
     so $T_s$ resolves $e,f$. Next, we check that there do not exist distinct vertices $u,v, w \ne s$ such that $T_s$ does not resolve $\left\{s,u\right\}$ and $\left\{v,w\right\}$. Indeed, in this case we must have $u \in T_s$, so $T_s$ must resolve $\left\{s,u\right\}$ from any edge that does not contain $u$, including $\left\{v,w\right\}$. Finally, we check that there do not exist distinct vertices $u,v \ne s$ such that $T_s$ does not resolve $\left\{s,u\right\}$ and $\left\{s,v\right\}$. Indeed, in this case we must have $u, v \in T_s$, so $T_s$ must resolve $\left\{s,u\right\}$ and $\left\{s,v\right\}$. 

     We check that there do not exist distinct vertices $u,v_1,v_2\ne s$ such that $T_s$ does not resolve $\{s,u\},\{u,v_1\}$ and $\{s,u\},\{u,v_2\}$. This is because that would imply that $T_s$ does not resolve $\{u,v_1\},\{u,v_2\}$, contradicting with our previous claim.

     If $T_s$ does not resolve $\{s,u\},\{u,v\}$, then we define $w_{s,u}=v$, and otherwise we define $w_{s,u}=s$. Let
     $$
     S''=S'\cup\bigcup_{s\in S,u\in N(s)}\{w_{s,u}\}.
     $$
     Thus, $S''$ is a fault-tolerant edge resolving set for $G$.

    For each $s\in S$, we have $|N(s)|\le 2^{|S|-1}$ by Theorem~\ref{thm:edimdeg} and thus, $|S''| \le |S|(1+2^{|S|-1})+|S|2^{|S|-1} = |S|(1+2^{|S|})$. 
\end{proof}

The next result shows that the bound in Theorem~\ref{thm:ftedim_upper} is sharp up to the base of the exponent.

\begin{thm}\label{thm:ftedim_lower}
    For every positive integer $k$, we have $\edim(H_k) = k$ and $\ftedim(H_k) \ge 2^{k-1}+1$.
\end{thm}

\begin{proof}
    Given a fault-tolerant edge resolving set $T$ for $H_k$, note that $T-\left\{u_1\right\}$ must be an edge resolving set for $H_k$. In $H_k$, consider the $2^k$ vertices which are neighbors of the center vertex $c$ in the copy of $K_{1,2^k}$. We partition these $2^k$ vertices into $2^{k-1}$ pairs of vertices, where the labels of the two vertices in each pair differ only in the first digit of their binary strings. 
    
    If vertices $v$ and $w$ are in the same pair, then the only vertices in $H_k$ that resolve $\left\{v,c\right\}$ and $\left\{w,c\right\}$ are $v$, $w$, and $u_1$. Indeed, the center vertex $c$ has distance $0$ to both $\left\{v,c\right\}$ and $\left\{w,c\right\}$. Every non-center vertex in the copy of $K_{1,2^k}$ that is not equal to $v$ or $w$ has distance $1$ to both $\left\{v,c\right\}$ and $\left\{w,c\right\}$. Finally, every vertex $u_i$ with $i \neq 1$ has distance $1$ to both $\left\{v,c\right\}$ and $\left\{w,c\right\}$ if $v$ and $w$ have $i^{\text{th}}$ digit $0$, and otherwise $u_i$ has distance $2$ to both $\left\{v,c\right\}$ and $\left\{w,c\right\}$. Thus, the only vertices in $H_k$ that resolve $\left\{v,c\right\}$ and $\left\{w,c\right\}$ are $v$, $w$, and $u_1$.
    
    So, in order for $T-\left\{u_1\right\}$ to be an edge resolving set for $H_k$, we must have $v \in (T-\left\{u_1\right\})$ or $w \in (T-\left\{u_1\right\})$. Thus, every pair in the partition must contribute at least one vertex to $T-\left\{u_1\right\}$. Therefore, $T-\left\{u_1\right\}$ must have size at least $2^{k-1}$.
\end{proof}

\section{Fault tolerance for adjacency dimension}\label{sec:adim}

If $G$ has order $1$, then $\adim(G) = 0$ and $\ftadim(G) = 1$. For both graphs $G$ of order $2$, it is easy to see that $\adim(G) = 1$ and $\ftadim(G) = 2$. In general, for any graph $G$ of order at least $2$, we must have $\adim(G) \ge 1$, so $\ftadim(G) \ge 2$. In the next result, we characterize the graphs $G$ of order at least $3$ such that $\adim(G) = 1$. This characterization has already been proved for connected graphs \cite{jannesari}; we prove it here for all graphs.

\begin{thm}\label{thm:adim_1}
    Let $G$ be a graph of order at least $3$. Then $\adim(G)=1$ if and only if $G$ has $3$ vertices and at least one edge, and $G$ is not a complete graph.
\end{thm}
\begin{proof}
The backward direction is easy to check, since there are only $2$ such graphs up to isomorphism, so we prove the forward direction. Suppose that $\adim(G)=1$, and let $S$ be an adjacency resolving set of size $1$. Since the distance of any vertex to the landmark in $S$ is capped at $2$, $G$ has at most $3$ vertices. Therefore $G$ has exactly $3$ vertices. If $G$ has no edge or if $G$ is isomorphic to $K_3$, then $\adim(G)\ge2$, a contradiction. Thus, $G$ is isomorphic to $P_3$ or the union of $P_2$ and a singleton.
\end{proof}

In the following result, we characterize the graphs $G$ that attain the minimum possible value of $\ftadim(G)$ among graphs $G$ of order at least $3$.

\begin{thm}
    Let $G$ be a graph of order at least $3$. Then $\ftadim(G)=2$ if and only if $G$ has $3$ vertices and at least one edge, and $G$ is not a complete graph.
\end{thm}
\begin{proof}
The backward direction is easy to check, so we prove the forward direction. Suppose that $\ftadim(G)=2$. Then by Lemma~\ref{lem:ftxdim}, we have $\adim(G) \le \ftadim(G)-1 = 1$. Since $G$ has order at least $3$, by Theorem~\ref{thm:adim_1} we can conclude that $G$ has $3$ vertices and at least one edge, and $G$ is not a complete graph.
\end{proof}

\begin{cor}
    For all graphs $G$, we have $\ftadim(G) = 2$ if and only if $\adim(G) = 1$.
\end{cor}

As with standard metric dimension, it is easy to see that $\ftadim(G) \le n$ for all graphs $G$ of order $n$, since any set of $n-1$ vertices in $G$ is an adjacency resolving set. The next result shows that this upper bound is sharp; the proof is the same as Theorem~\ref{thm:ftdim<n}.

\begin{thm}\label{thm:ftadim<n}
    For all graphs $G$ of order $n$, we have $\ftadim(G) < n$ if and only if there exists $v \in V(G)$ such that for all $u \in V(G)$ with $u \neq v$ we have $N(u)-\left\{v\right\} \neq N(v)-\left\{u\right\}$.
\end{thm}

\begin{cor}
    For all graphs $G$ of order $n$, we have $\ftadim(G) = n$ if and only if for all $v \in V(G)$ there exists $u \in V(G)$ with $u \neq v$ such that $N(u)-\left\{v\right\} = N(v)-\left\{u\right\}$.
\end{cor}

For multiple results in this section, we will use the following construction of an infinite family of connected graphs $A_1, A_2, \dots$ that are very similar to the family $H_1, H_2, \dots$ that was defined in Section~\ref{sec:edim}. Start with a copy of $H_k$, and then add edges between every pair of vertices $u_i, u_j$ for $1 \le i < j \le k$, so that the vertices $u_1, u_2, \dots, u_k$ form a clique of size $k$. Delete any vertices with the same adjacency vector as the center vertex $c$ with respect to $S = \left\{u_1, \dots, u_k\right\}$. Call the resulting graph $A_k$. 

With similar arguments, the only deleted vertex is labeled with the all-$1$ binary string, so $A_k$ has order $k+2^k$. Moreover, $\adim(A_k) \le k$ since $S$ is an adjacency resolving set for $A_k$. We must also have $\adim(A_k) \ge k$, since Geneson and Yi proved that any graph $G$ with $\adim(G) = x$ must have order at most $x+2^x$ \cite{gy}. Thus, $\adim(A_k) = k$. Finally, note that every element of $S$ has degree $2^{k-1}+k-1$, since every element of $S$ is neighbors with all other elements of $S$, and the vertex that we delete to create $A_k$ was not a neighbor of any element of $S$ in $H_k$.

Geneson and Yi \cite{gy} showed that the maximum possible degree of any vertex in a graph $G$ with $\adim(G) = k$ is $2^k+k-1$. In the next result, we determine the maximum possible degree of any vertex in an adjacency resolving set of size $k$.

\begin{thm}\label{thm:adimdeg}
    Any vertex in an adjacency resolving set of size $k$ must have degree at most $2^{k-1}+k-1$. Moreover, this bound is sharp.
\end{thm}

\begin{proof}
    For the upper bound, consider a graph $G$ with an arbitrary vertex $v$ in an adjacency resolving set $X$ of size $k$. Without loss of generality, suppose that the first coordinate of $a_X(u)$ corresponds to $v$ for any vertex $u$ in $G$. For any neighbor $u$ of $v$, the first coordinate of $a_X(u)$ must be $1$. At most $k-1$ neighbors of $v$ may have $0$ in one of their coordinates besides the first coordinate. All other neighbors of $v$ may only $1$ or $2$ in all of their coordinates besides the first coordinate. Thus, $v$ has at most $k-1+2^{k-1}$ neighbors.

    For the lower bound, we use the graph $A_k$ which was defined at the beginning of the section. Recall that $\adim(A_k) = k$, $S = \left\{u_1, \dots, u_k\right\}$ is an adjacency resolving set for $A_k$, and every element of $S$ has degree $2^{k-1}+k-1$. Thus, our bound is sharp.
\end{proof}

\begin{cor}
The maximum possible minimum degree of a graph $G$ with $\adim(G)=k$ is at most $2^{k-1}+k-1$.    
\end{cor}

In the next result, we obtain an upper bound on fault-tolerant adjacency dimension in terms of adjacency dimension.

\begin{thm}\label{thm:ftadim_upper}
    For every graph $G$ of order greater than $1$, we have $\ftadim(G) \le \adim(G)+2^{\adim(G)}$.
\end{thm}

\begin{proof}
    Geneson and Yi \cite{gy} proved that every graph $G$ with $\adim(G) = k$ has order at most $k+2^k$. Clearly, the whole vertex set is a fault-tolerant adjacency resolving set for any graph $G$, so  $\ftadim(G) \le \adim(G)+2^{\adim(G)}$.
\end{proof}

The next result shows that Theorem~\ref{thm:ftadim_upper} is sharp up to a constant factor.

\begin{thm}
    For every positive integer $k$, we have $\adim(A_k) = k$ and $\ftadim(A_k) \ge 2^{k-1}$.
\end{thm}

\begin{proof}
It remains to prove that $\ftadim(A_k) \ge 2^{k-1}$ for every positive integer $k$.
    Given a fault-tolerant adjacency resolving set $T$ for $A_k$, note that $T-\left\{u_1\right\}$ must be an adjacency resolving set for $H_k$. In $A_k$, consider the $2^k-1$ vertices which are neighbors of the center vertex $c$ in the copy of $K_{1,2^k-1}$. We partition these $2^k-1$ vertices into $2^{k-1}-1$ pairs of vertices and a single remaining vertex, where the labels of the two vertices in each pair differ only in the first digit of their binary strings. 
    
    If vertices $v$ and $w$ are in the same pair, then the only vertices in $A_k$ that $1$-resolve $v$ and $w$ are $v$, $w$, and $u_1$. Indeed, the center vertex $c$ has distance $1$ to both $v$ and $w$. Every non-center vertex in the copy of $K_{1,2^k-1}$ that is not equal to $v$ or $w$ has distance $2$ to both $v$ and $w$. Finally, every vertex $u_i$ with $i \neq 1$ has distance $1$ to both $v$ and $w$ if $v$ and $w$ have $i^{\text{th}}$ digit $0$, and otherwise $u_i$ has distance $2$ to both $v$ and $w$. Thus, the only vertices in $A_k$ that $1$-resolve $v$ and $w$ are $v$, $w$, and $u_1$.
    
    So, in order for $T-\left\{u_1\right\}$ to be an adjacency resolving set for $A_k$, we must have $v \in (T-\left\{u_1\right\})$ or $w \in (T-\left\{u_1\right\})$. Thus, every pair in the partition must contribute at least one vertex to $T-\left\{u_1\right\}$. Therefore, $T-\left\{u_1\right\}$ must have size at least $2^{k-1}-1$.
\end{proof}

\begin{cor} \label{cor:adim} We have
    \[\lim_{k \rightarrow \infty} \left( \max_{G: \text{ } \adim(G) = k} \ftadim(G) \right) = \Theta(2^k).\]
\end{cor}

If we restrict the maximum in the corollary so that it only ranges over connected graphs $G$ with $\adim(G) = k$, rather than over all graphs $G$ with $\adim(G) = k$, then we still obtain the same limit because the infinite family of graphs $A_1,A_2,\dots$ used in the proof of Theorem~\ref{thm:adimdeg} are connected.

\section{Fault tolerance for truncated metric dimension}\label{sec:ftdimk}

If $G$ has order $1$, then $\dim_k(G) = 0$ and $\ftdim_k(G) = 1$. However, for any graph $G$ of order at least $2$, we must have $\dim_k(G) \ge 1$ and $\ftdim_k(G) \ge 2$. We characterize the graphs with $\dim_k(G)=1$ in the next result. 

\begin{thm}\label{thm:dimk_1}
For all graphs $G$ and positive integers $k$, we have $\dim_k(G)=1$ if and only if $G$ is either $P_i$ for some $2 \le i \le k+2$ or the union of $P_j$ and a singleton for some $1 \le j \le k+1$.
\end{thm}
\begin{proof}
Suppose that $G$ is either $P_i$ for some $2 \le i \le k+2$ or the union of $P_j$ and a singleton for some $1 \le j \le k+1$. Since $G$ has order at least $2$, we must have $\dim_k(G) > 0$. Since either endpoint of the $P_i$ or $P_j$ is a $k$-resolving set for $G$, we have $\dim_k(G) = 1$.

For the other direction, suppose that $G$ is a graph with $\dim_k(G)=1$, and let $S = \{s\}$ be a $k$-resolving set of size $1$. Since the truncated distance of any vertex to $s$ is capped at $k+1$, $G$ has order at most $k+2$. Moreover, $G$ has order at least $2$, or else we would have $\dim_k(G) = 0$.

The vertex $s$ must have degree at most $1$, or else it would not be able to $k$-resolve its neighbors. Every vertex $v$ of $G$ must have degree at most $2$, or else $s$ would not be able to $k$-resolve the neighbors of $v$. No connected component of $G$ can be a cycle, or else at least two landmarks would be required to $k$-resolve $G$. Thus, every connected component of $G$ is a path. 

$G$ has at most $2$ connected components, or else $s$ would be unable to resolve two vertices that are in different connected components from $s$ and from each other. Finally, $G$ has at most $1$ connected component of order at least $2$, or else there would be a pair of vertices that are not in the same connected component as $s$, and thus, cannot be $k$-resolved by $s$. Therefore, $G$ is either $P_i$ for some $2 \le i \le k+2$ or the union of $P_j$ and a singleton for some $1 \le j \le k+1$.
\end{proof}

In the following result, we characterize the graphs $G$ that attain the minimum possible value of $\ftdim_k(G)$ among graphs of order at least $2$. 

\begin{thm}
For all graphs $G$ and positive integers $k$, we have $\ftdim_k(G)=2$ if and only if $G$ is either $P_i$ for some $2 \le i \le k+2$ or the union of $P_j$ and a singleton for some $1 \le j \le k+1$.
\end{thm}
\begin{proof}
Suppose that $G$ is either $P_i$ for some $2 \le i \le k+2$ or the union of $P_j$ and a singleton for some $1 \le j \le k+1$. Since $G$ has order at least $2$, we must have $\ftdim_k(G) > 1$. If $G$ is the union of two singletons, then let $S = V(G)$. Otherwise, if $G$ contains a path $P$ with at least $2$ vertices, then let $S$ consist of the two endpoints of $P$. In either case, $S$ is clearly a fault-tolerant $k$-resolving set for $G$, so we have $\ftdim_k(G) = 2$.

For the other direction, suppose that $\ftdim_k(G)=2$. Then $G$ has order at least $2$. By Lemma~\ref{lem:ftxdim}, we have $\dim_k(G) \le \ftdim_k(G)-1 = 1$. Thus, by Theorem~\ref{thm:dimk_1} we can conclude that $G$ is either $P_i$ for some $2 \le i \le k+2$ or the union of $P_j$ and a singleton for some $1 \le j \le k+1$.
\end{proof}

\begin{cor}
    For all graphs $G$ and positive integers $k$, we have $\ftdim_k(G) = 2$ if and only if $\dim_k(G) = 1$.
\end{cor}

As with standard metric dimension and adjacency dimension, it is easy to see for all $k \ge 1$ that $\ftdim_k(G) \le n$ for all graphs $G$ of order $n$, since any set of $n-1$ vertices in $G$ is a $k$-truncated resolving set. The next result shows that this upper bound is sharp and generalizes Theorem~\ref{thm:ftadim<n}; the proof is the same as Theorem~\ref{thm:ftdim<n}.

\begin{thm}
    For all $k \ge 1$ and graphs $G$ of order $n$, we have $\ftdim_k(G) < n$ if and only if there exists $v \in V(G)$ such that for all $u \in V(G)$ with $u \neq v$ we have $N(u)-\left\{v\right\} \neq N(v)-\left\{u\right\}$.
\end{thm}

\begin{cor}
    For all $k \ge 1$ and graphs $G$ of order $n$, we have $\ftdim_k(G) = n$ if and only if for all $v \in V(G)$ there exists $u \in V(G)$ with $u \neq v$ such that $N(u)-\left\{v\right\} = N(v)-\left\{u\right\}$.
\end{cor}

Before proving some extremal bounds for $\ftdim_k(G)$ with respect to $\dim_k(G)$, we first discuss a construction from Geneson et al \cite{gkl} which was used to obtain a number of extremal results about metric dimension and pattern avoidance, solving several open problems from \cite{gmd}.

\begin{definition}\label{def:dk}
    Let $D_k$ be the graph that has vertex set $\mathbb{Z}_{\ge 0}^k$ with edges between every pair of vertices that differ by at most $1$ in each coordinate. 
\end{definition}

For the next result in this section, we use the following construction which is similar to one used in \cite{hernando} for bounding the order of graphs in terms of metric dimension and diameter. Let $A_k(q)$ denote the set of vertices in $D_k$ contained in the $k$-cube $[q,3q]^k$. Let $M_{k,0}(q)$ denote the set of $k$ vertices with one coordinate $0$ and all others equal to $2q$. For each $i$ with $0 \le i \le q-2$ in increasing order, let $M_{k,i+1}(q)$ denote the set of vertices with one coordinate equal to $i$ which are neighbors of some vertex in $M_{k,i}(q)$. Let \[R_k(q) = A_k(q) \cup \bigcup_{i = 0}^{q-1} M_{k,i}(q).\] The induced subgraph $I_k(q)$ of $D_k$ on the vertices of $R_k(q)$ has metric dimension at most $k$, and the coordinates of its vertices correspond to distance vectors for the resolving set consisting of the points in $M_{k,0}(q)$. In fact the metric dimension of $I_k(q)$ is $k$, because it has vertices with degree $3^k-1$ and it is known \cite{gkl} that the maximum possible degree of a vertex in a graph of metric dimension $k$ is $3^k-1$. For each $j$ with $1\le j\le q-1$, each landmark $v$ has $(2j+1)^{k-1}$ vertices $u \neq v$ with $\dist(u, v) = j$.

\begin{thm}\label{thm:dimkdeg}
    Any vertex in a $k$-resolving set of size $j$ has degree at most $3^{j-1}$. Moreover, this bound is sharp for $k \ge 2$.
\end{thm}

\begin{proof}
    The upper bound follows from \cite{gkl} since any $k$-resolving set of size $j$ is also a resolving set of size $j$. To see that this bound is sharp for $k \ge 2$, consider the graph $I_k(1)$. All coordinates are at most $3$, so the resolving set $M_{k,0}(1)$ is also a $k$-resolving set.
\end{proof}

\begin{cor}
The maximum possible minimum degree of a graph $G$ with a $k$-resolving set of size $j$ is at most $3^{j-1}$.    
\end{cor}

Now, we are ready to obtain an upper bound on fault-tolerant $k$-truncated metric dimension in terms of $k$-truncated metric dimension.

\begin{thm}\label{thm:ftdimk_upper}
    For every positive integer $k$ and graph $G$ of order greater than $1$, we have $\ftdim_k(G) \le \dim_k(G)(2+3^{\dim_k(G)-1})$.
\end{thm}

\begin{proof}
    Given a $k$-resolving set $S$ for $G$, define \[S' = S \cup \bigcup_{s \in S} N(s).\] For each $s \in S$, let $T_s = S' - \left\{s\right\}$.
    
    We claim that $T_s$ $k$-resolves all pairs of vertices in $G$, except possibly $s$ and one other vertex $u$. First, we show that if $u \neq s$ and $v \neq s$ then $T_s$ $k$-resolves $u, v$. Suppose that $T_s$ does not $k$-resolve them. Given that $S$ $k$-resolves them, without loss of generality assume that $\min(k+1,\dist(s,u))<\min(k+1,\dist(s,v))$. So, we have $\dist(s,u) < k+1$ and $\dist(s,u) < \dist(s,v)$.

Clearly $s$ is not isolated. Since for any vertex $x\ne s$, \[\dist(s, x) = 1+\min_{x' \in N(s)} \dist(x', x),\]
    assuming that $\argmin_{y\in N(s)}\dist(y,u)=u'$  and $\argmin_{y\in N(s)}\dist(y,v)=v'$ we have $\dist(u',u) < k$ and
     \[\min(k+1, \dist(u', u)) = \dist(u', u) < \min(k+1,\dist(v',v))\le \min(k+1,\dist(u',v)).\] Thus, $T_s$ $k$-resolves $u$ and $v$. Next, we show that there do not exist distinct vertices $u,v\ne s$ such that $T_s$ does not $k$-resolve $s,u$ and $s,v$. If there exist such vertices $u,v$, then $u,v\notin T_s$ and for any vertex $x\in T_s$ we have $\min(k+1,\dist(u,x))=\min(k+1,\dist(s,x))=\min(k+1,\dist(v,x))$, i.e., $T_s$ does not $k$-resolve $u,v$. This contradicts our previous claim.
    
    If $T_s$ does not $k$-resolve $s$ from some other vertex $u$, then define $u_s = u$ and otherwise we define $u_s = s$. Let \[S'' = S' \cup \bigcup_{s \in S} \left\{u_s\right\}.\]
    
    Clearly, $S''$ is a $k$-resolving set for $G$. Consider an arbitrary $t \in S''$, and let $X = S''-\left\{t\right\}$. If $t \not \in S$, then $X$ clearly $k$-resolves all pairs of vertices in $G$ since $S \subseteq X$. Otherwise, suppose that $t \in S$. Recall that $T_t$ $k$-resolves all pairs of vertices in $G$, except possibly $t$ and $u_t$. However, $u_t \in X$ and $T_t \subseteq X$, so $X$ $k$-resolves all pairs of vertices in $G$. Thus, $S''$ is a fault-tolerant $k$-resolving set for $G$.

    For each $s\in S$, we have $|N(s)|\le 3^{|S|-1}$ by Theorem~\ref{thm:dimkdeg} and thus, $|S''| \le |S|(2+3^{|S|-1})$. 
\end{proof}

In the next result, we use the infinite family of connected graphs $J_1, J_2, J_3, \dots$ which was defined in Section~\ref{sec:ftdim}.

\begin{thm}\label{thm:ftdimk_lower}
    For all integers $k \ge 2$ and $j \ge 1$, we have $\dim_k(J_j) = j$ and $\ftdim_k(J_j) \ge 3^{j-1}-j$.
\end{thm}

\begin{proof}
    Since $\dim_k(G) \ge \dim(G)$ for all graphs $G$ and positive integers $k$, we have $\dim_k(J_j) \ge \dim(J_j) = j$. Note that all elements of $S$ in $J_j$ have distance at most $3$ to every vertex in $J_j$, so $S$ is a $k$-resolving set for $J_j$, and $\dim_k(J_j) \le j$. Thus, $\dim(J_j) = j$. 

    By Theorem~\ref{thm:ftdim_lower}, we have $\ftdim_k(J_j) \ge \ftdim(J_j) \ge 3^{j-1}-j$.
\end{proof}

\begin{cor}
    For all integers $k \ge 2$, we have \[\lim_{j \rightarrow \infty} \left( \max_{G: \text{ } \dim_k(G) = j} \frac{\log_3(\ftdim_k(G))}{j} \right) = 1.\] However, \[\lim_{j \rightarrow \infty} \left( \max_{G: \text{ } \dim_1(G) = j} \frac{\log_2(\ftdim_1(G))}{j} \right) = 1.\]
\end{cor}

\begin{proof}
    The first limit follows from the bounds in this section, while the second limit follows from Corollary~\ref{cor:adim}.
\end{proof}

As with the other variants, if we restrict the maximum in the corollary so that it only ranges over connected graphs $G$ with $\dim_k(G) = j$, rather than over all graphs $G$ with $\dim_k(G) = j$, then we still obtain the same limit because the infinite family of graphs $J_1,J_2,\dots$ used in the proof of Theorem~\ref{thm:ftdimk_lower} are connected.

\section{Fault tolerance for local metric dimension}\label{sec:ftldim}

If $G$ has no edges, then clearly we have $\ldim(G) = 0$. If $G$ has any edge $\left\{u,v\right\}$, then clearly $\ldim(G) > 0$, or else $u$ and $v$ could not be resolved. We claim that $\ldim(G) = 1$ if and only if $G$ is a union of a connected bipartite graph with at least one edge and any number of singletons. This is already known for the case of connected graphs \cite{okamoto}; we include a proof of the general case here for completeness.

\begin{thm}\label{thm:ldim_1}
    For all graphs $G$, $\ldim(G) = 1$ if and only if $G$ is a union of a connected bipartite graph with at least one edge and any number of singletons.
\end{thm}

\begin{proof}
    If $G$ is a union of a connected bipartite graph $H$ with at least one edge and any number of singletons, then we can focus just on $H$ since none of the singletons have to be distinguished from any other vertices. If we select any vertex $v$ in $H$, then it is a local resolving set for $G$. Indeed, for any edge $\left\{u,w\right\}$ in $H$, the distances of $u$ and $w$ to $v$ must have different parities.

    For the other direction, suppose that $\ldim(G) = 1$. Then, $G$ must have some edge, or else $\ldim(G) = 0$. Furthermore, $G$ must have at most one connected component $H$ with an edge, or else $\ldim(G) > 1$. Now, we claim that $H$ must be bipartite. Indeed, since $\ldim(G) = 1$, there must exist some vertex $v$ which distinguishes all pairs of adjacent vertices in $G$. We define two sets $L$ and $R$. In particular, for each vertex $u$ in $H$, we place $u$ in $L$ if it has even distance to $v$, and otherwise we place $u$ in $R$. If there are two vertices $u$ and $w$ in $H$ with the same distance to $v$, then they are not adjacent, since $v$ does not distinguish them. If $u$ and $w$ have different distances to $v$ of the same parity, then they are not adjacent, since otherwise their distances to $v$ would be at most $1$ apart. So, no pair of vertices in $L$ are adjacent, and no pair of vertices in $R$ are adjacent. Thus, $H$ is bipartite. 
\end{proof}

As for the maximum possible values of $\ldim(G)$, it is easy to see that $\ldim(G) \le n-1$ for all graphs $G$ of order $n$, since any set of $n-1$ vertices in $G$ is clearly a local resolving set. Moreover, we have $\ldim(G) \le n-2$ for all graphs $G$ of order $n$ that are not complete, since any set of $n-2$ vertices in $G$ is a local resolving set if the two vertices that are not in the set are not adjacent. 

\begin{remark}
    Among graphs $G$ of order $n$, we have that $\ldim(G) = n-1$ if and only if $G$ is a complete graph \cite{okamoto}. Indeed, if $G$ is a complete graph of order $n$, then $\ldim(G) \ge n-1$ since any set of vertices in $G$ of size $n-2$ would not distinguish the two vertices that are not in the set. Thus, $\ldim(G) = n-1$ if $G$ is a complete graph of order $n$. In the other direction, if $G$ is a graph of order $n$ with $\ldim(G) = n-1$, then all pairs of vertices in $G$ must be adjacent, so $G$ is complete. This result was already proved for connected graphs \cite{okamoto}, but the general proof is the same.
\end{remark} 

Now, we turn to fault tolerance for local metric dimension. If $G$ is a graph with no edges, then we have $\ftldim(G) = 1$, since there are no adjacent pairs of vertices. Observe that this situation is very different from fault tolerance for standard metric dimension, since we have $\ftdim(G) = n$ for graphs $G$ of order $n$ with no edges. If $G$ has at least one edge, then $\ldim(G) \ge 1$, so $\ftldim(G) \ge 2$. The following characterization implies that this bound is sharp.

\begin{thm}
    For any graph $G$, $\ftldim(G) = 2$ if and only if $G$ is a union of a connected bipartite graph with at least one edge and any number of singletons. 
\end{thm}

\begin{proof}
    First, suppose that $G$ is a union of a connected bipartite graph $H$ with at least one edge and any number of singletons. Then, $H$ has at least two vertices and at least one edge. Let $S$ be any set of two vertices from $H$. Any vertex in $H$ is a local resolving set for $G$, so $S$ is a fault-tolerant local resolving set for $G$. Thus, $\ftldim(G) \le 2$. Since $G$ has at least one edge, we also have $\ftldim(G) \ge 2$, so $\ftldim(G) = 2$. 
    
    For the other direction, suppose that we have an arbitrary graph $G$ with $\ftldim(G) = 2$. Since $\ftldim(G) > 1$, $G$ has at least one edge. By Lemma~\ref{lem:ftxdim}, $\ldim(G) \le \ftldim(G)-1 = 1$. Thus, by Theorem~\ref{thm:ldim_1}, $G$ must be the union of a connected bipartite graph with at least one edge and any number of singletons. 
\end{proof}

\begin{cor}
    For all graphs $G$, we have $\ftldim(G) = 2$ if and only if $\ldim(G) = 1$.
\end{cor}

As for the maximum possible value of $\ftldim(G)$ with respect to the order of $G$, it is easy to see that $\ftldim(G) \le n$ for all graphs $G$ of order $n$, since any set of $n-1$ vertices in $G$ is a local resolving set. The next result shows that this upper bound is sharp.

\begin{thm}
    For all graphs $G$ of order $n$, we have $\ftldim(G) < n$ if and only if there exists $v \in V(G)$ such that for all $u \in V(G)$ with $\left\{u, v\right\} \in E(G)$ we have $N(u)-\left\{v\right\} \neq N(v)-\left\{u\right\}$.
\end{thm}

\begin{proof}
Suppose that there exists a vertex $v \in V(G)$ such that for all $u \in V(G)$ with $\left\{u, v\right\} \in E(G)$ we have $N(u)-\left\{v\right\} \neq N(v)-\left\{u\right\}$. Let $S = V(G)-\left\{v\right\}$. Then for all $s \in S$ with $\left\{s,v\right\} \in E(G)$, the set $S-\left\{s\right\}$ distinguishes $v$ from $s$, since $N(v)-\left\{s\right\} \neq N(s)-\left\{v\right\}$. Thus, $\ftldim(G) < n$.

For the other direction, suppose that $\ftldim(G) < n$. Then there exists a vertex $v \in V(G)$ such that $V(G)-\left\{v\right\}$ is a fault-tolerant local resolving set. Then for every vertex $u \in V(G)-\left\{v\right\}$, the set $V(G)-\left\{u,v\right\}$ is a local resolving set. In particular, $V(G)-\left\{u,v\right\}$ resolves $u$ and $v$ if $u$ and $v$ are adjacent. So, if $u$ and $v$ are adjacent, then there must exist some vertex $w \in V(G)-\left\{u,v\right\}$ such that $w$ is adjacent to $u$ or $v$ but not both, and thus, $N(u)-\left\{v\right\} \neq N(v)-\left\{u\right\}$.
\end{proof}

\begin{cor}\label{cor:ftldim_n}
    For all graphs $G$ of order $n$, we have $\ftldim(G) = n$ if and only if for all $v \in V(G)$ there exists $u \in V(G)$ with $\left\{u, v\right\} \in E(G)$ such that $N(u)-\left\{v\right\} = N(v)-\left\{u\right\}$.
\end{cor}

By Corollary~\ref{cor:ftldim_n}, it is clear that we have $\ftldim(K_n) = n$ for all positive integers $n$. If $G$ is a $1$-regular, i.e., if $G$ is a union of disconnected edges, then we also have $\ftldim(G) = n$ by the same corollary. For another example, consider the graph $G$ with vertices $v_i, u_i$ for every $1 \le i \le k$. Suppose that $v_i$ is adjacent to $u_i$ for every $1 \le i \le k$, and both $u_i$ and $v_i$ are adjacent to $v_{i+1}$ and $u_{i+1}$ for every $1 \le i \le k-1$. Then, we also have $\ftdim(G) = n$ by Corollary~\ref{cor:ftldim_n}.

We have discussed the characterization of the graphs $G$ of order $n$ with $\ldim(G) = 1$ and $\ldim(G) = n-1$, and we also characterized the graphs $G$ with $\ftldim(G) = 2$ and $\ftldim(G) = n$. A remaining problem is to complete the characterization for other values of $\ldim(G)$ and $\ftldim(G)$.

We showed for all graphs $G$ that $\ldim(G) = 1$ if and only if $\ftldim(G) = 2$. Moreover, clearly we have $\ldim(G) = 0$ if and only if $\ftldim(G) = 1$. Thus, there is a finite upper bound on $\ftldim(G)$ for any graph $G$ with $\ldim(G) \le 1$. However, based only on known results, it is plausible that in general, $\ftldim(G)$ might not have a finite upper bound in terms of $\ldim(G)$. In particular, does there exist a family of graphs $G$ with $\ldim(G) = 2$ such that $\ftldim(G)$ is unbounded on the family?

In all of our bounds for variants of metric dimension $\xdim(G)$, the maximum possible value of $\ftxdim(G)$ with respect to $\xdim(G)$ grows nearly on the same order as the maximum possible value of $\Delta(G)$ with respect to $\xdim(G)$. However, there is no upper bound on $\Delta(G)$ for graphs $G$ with $\ldim(G) = k$. In particular, for any positive integers $k$ and $r$, it is easy to construct connected graphs $G$ with $\ldim(G) = k$ and $\Delta(G) \ge r$. Indeed, we can start with any connected graph $G$ that has $\ldim(G) = k$ and apply the following lemma.

\begin{lem}
    Suppose that $G$ is connected and $\ldim(G) = k \ge 1$. Pick a vertex $v$ in $G$, and then add $r$ new vertices that are neighbors only with $v$. The resulting graph $G'$ has $\ldim(G') = k$ and $\Delta(G') \ge r$.
\end{lem}

\begin{proof}
    The fact that $\Delta(G') \ge r$ follows from definition. To see that $\ldim(G') = k$, we prove that $\ldim(G') \le k$ and $\ldim(G') \ge k$. To prove that $\ldim(G') \le k$, it suffices to observe that any local resolving set for $G$ is also a local resolving set for $G'$. Indeed, note that every vertex in $G$ can resolve all of the new pairs of adjacent vertices in $G'$.
    
    Now, we prove that $\ldim(G') \ge k$. Suppose for contradiction that $G'$ has a local resolving set $S$ with $|S| < k$. Then $S$ must contain some new vertex in $G'$, since otherwise $S$ would be a local resolving set for $G$, which would violate the fact that $\ldim(G) = k$. Let $S'$ be obtained from $S$ by removing any new vertices and adding $v$. Then $S'$ is also a local resolving set for $G'$, and $|S'| \le |S| < k$. Since $S'$ only has vertices from $G$, $S'$ is also a local resolving set for $G$, but $|S'| < k$ and $\ldim(G) = k$. This is a contradiction, so $G'$ does not have a local resolving set $S$ with $|S| < k$. Thus, $\ldim(G') \ge k$.
\end{proof}

\section{More extremal results for standard metric dimension}\label{sec:dim_extremal}

Geneson et al \cite{gkl} proved the following lemma, which they used to obtain several exact extremal results about metric dimension.

\begin{lem}\label{dk_lem}\cite{gkl}
   For every graph $H$ and positive integer $k$, there exists a graph of metric dimension $k$ which contains $H$ if and only if $D_k$ contains $H$.  
\end{lem}

We use this fact in the next proof to generalize the result from \cite{gkl} that the maximum possible degree of a vertex in a graph of metric dimension $k$ is $3^k - 1$.

\begin{thm}
    For any vertex $v$ in a graph of metric dimension $k$, the maximum possible number of vertices $u \neq v$ with $\dist(u, v) \le j$ is equal to $(2j+1)^{k}-1$.
\end{thm}

\begin{proof}
    Suppose that $G$ is a graph of metric dimension $k$. Let $S$ be a resolving set of $G$ with $|S| = k$, and embed $G$ into $D_k$ with respect to some ordering of $S$. For any vertices $u \neq v$ in $G$ with $\dist(u, v) \le j$, all coordinates of $u$ and $v$ in the embedding must differ by at most $j$. Thus, for any fixed vertex $v$, there are at most $(2j+1)^k-1$ vertices $u$ with $\dist(u, v) \le j$.
    To see that this bound can be attained, consider an induced subgraph $H$ of $D_k$ with order $(2j+1)^k$ on the vertices with all coordinates at most $2j$. The distance of the vertex $v$ with all-$j$ coordinates to all other $(2j+1)^k-1$ vertices of $H$ is at most $j$, and by Lemma~\ref{dk_lem} there is a graph $H'$ of metric dimension $k$ that contains $H$. Hence in $H'$ the number of vertices whose distance to $v$'s counterpart $v'$ do not exceed $j$, excluding $v'$ itself, is at least $(2j+1)^k-1$.
\end{proof}

In order to get the past upper bounds on the fault-tolerant metric dimension and the $(k+1)$-metric dimension, the papers \cite{ftdim08} and \cite{ftdim24} used an upper bound on the maximum possible number of vertices with distance at most $j$ to a given landmark. 

\begin{lem}\label{old_distj_bound} \cite{ftdim08}
    For any vertex $v$ in a resolving set of size $k$, the possible number of vertices $u \neq v$ with $\dist(u, v) \le j$ is at most $1+j(2j+1)^{k-1}$.
\end{lem}

In the following results, we obtain an improvement on the bound in Lemma~\ref{old_distj_bound}. In particular, we determine the exact value of the maximum possible number of vertices $u \neq v$ with $\dist(u, v) \le j$.

\begin{thm}
    For any vertex $v$ in a resolving set of size $k$, the maximum possible number of vertices $u \neq v$ with $\dist(u, v) = j \ge 1$ is equal to $(2j+1)^{k-1}$. 
\end{thm}

\begin{proof}
    For the upper bound, suppose that $G$ is a graph of metric dimension $k$. Let $S$ be a resolving set of $G$ with $|S| = k$, and embed $G$ into $D_k$ with respect to some ordering of $S$. Let $v$ be a landmark vertex, so one of its coordinates (say the $i^{\text{th}}$) is $0$ and the rest are positive. For any vertex $u \neq v$ in $G$ with $\dist(u, v) = j \ge 1$, all coordinates of $u$ and $v$ in the embedding must differ by at most $j$. The $i^{\text{th}}$ coordinate must be equal to $j$. Thus, there are at most $(2j+1)^{k-1}$ vertices $u$ with $\dist(u, v) = j$. 

    For the lower bound, we use the graphs $I_k(q)$ constructed in Section~\ref{sec:ftdimk} with sufficiently large $q$. As noted before, each landmark $v$ in $M_{k,0}(q)$ has $(2j+1)^{k-1}$ vertices $u \neq v$ in $I_k(q)$ with $\dist(u, v) = j$.
\end{proof}

Using the upper bound from the last result together with the same lower bound construction, we obtain the following corollary.

\begin{cor}
    For any vertex $v$ in a resolving set of size $k$, the maximum possible number of vertices $u \neq v$ with $\dist(u, v) \le j$ is equal to \[\sum_{i = 1}^j (2i+1)^{k-1} = \frac{2^{k-1}}{k}j^k (1+o(1)).\]
\end{cor}

\section{An equivalence between a problem in extremal set theory and a problem about edge metric dimension}\label{sec:equiv}

Geneson \cite{gmd} investigated the maximum possible clique number of a graph with edge metric dimension at most $k$. He showed that the answer is at most $O(2^{k/2})$ and at least $k+1$. The upper bound follows since there are at most $2^k$ possible distance vectors for the $\binom{n}{2}$ edges of a clique of size $n$ in a graph of edge metric dimension $k$, and every edge in the clique must have a unique distance vector. Later, Geneson et al \cite{gkl} used the probabilistic method to show the existence of graphs of edge metric dimension at most $k$ with cliques of size $\Omega((\frac{8}{3})^{k/6})$. Let $\mc(k)$ denote the maximum possible clique number of a graph with edge metric dimension at most $k$. Then, we have $\mc(k) = O(2^{k/2})$ and $\mc(k) = \Omega((\frac{8}{3})^{k/6})$.

Let $\ek(k)$ denote the maximum possible size of a family of subsets of $\left\{1,2,\dots,k\right\}$ such that all pairwise unions are distinct. The study of $\ek(k)$ was initiated by Erd\H{o}s and Kleitman \cite{ek74}, and they claimed that there exist $0 < \epsilon_1 < \epsilon_2 < 1$ such that $(1+\epsilon_1)^k < \ek(k) < (1+\epsilon_2)^k$ for all $k$ sufficiently large. They offered $\$25$ for finding $\epsilon_1, \epsilon_2$ with $\epsilon_2 / \epsilon_1 \le 1.01$. 

\begin{thm}
    For all positive integers $k$, we have $\mc(k) = \ek(k)$.
\end{thm}

\begin{proof}
    Let $\mathcal{F} = \left\{S_1,\dots,S_{\ek(k)}\right\}$ be a family of subsets of $\left\{1,2,\dots,k\right\}$ of size $\ek(k)$ such that all pairwise unions of elements of $\mathcal{F}$ are distinct. From $\mathcal{F}$, we construct a graph $G$. $G$ has $k+\ek(k)$ vertices $u_1,\dots,u_k$ and $v_1,\dots,v_{\ek(k)}$. The vertices $v_1,\dots,v_{\ek(k)}$ form a clique. There are also edges between $u_i$ and $v_j$ for every pair $i, j$ such that $i$ is an element of the subset $S_j$. 

    For each edge in $G$, we consider the distance vector with respect to $u_1, \dots, u_k$. For any edge of the form $\left\{u_i, v_j\right\}$, there will be a $0$ in the coordinate for $u_i$, but all edges of the form $\left\{v_i,v_j\right\}$ will have no $0$ coordinates. All edges of the form $\left\{v_i,v_j\right\}$ have $1$ in the coordinate corresponding to $u_t$ if $t \in S_i$ or $t \in S_j$, and otherwise $\left\{v_i,v_j\right\}$ has $2$ in the coordinate corresponding to $u_t$. Since all pairwise unions in $\mathcal{F}$ are distinct, all edges of the form $\left\{v_i,v_j\right\}$ have distinct distance vectors. 

    For any edge of the form $\left\{u_i, v_j\right\}$, there will be a $0$ in the coordinate for $u_i$, there will be a $1$ in the coordinate for each $u_t$ such that $t \neq i$ and $t \in S_j$, and there will be a $2$ in the coordinate for each $u_t$ such that $t \neq i$ and $t \not \in S_j$.
    Every pair of edges $\{u_i,v_j\},\{u_{i'},v_{j'}\}$ where $i\ne i'$ have different distance vectors because their $0$ coordinates are at different positions $i$ and $i'$. For any pair of edges $\{u_i,v_j\},\{u_i,v_{j'}\}$ with $j\ne j'$, let their distance vectors be $x$ and $y$, respectively. Both vectors have a single $0$-coordinate at position $i$, so we transform them to $x'$ and $y'$ by replacing the $0$-coordinate at position $i$ with a $1$. Hence, $x'$ has ones only at positions $s\in S_j$, and $y'$ has ones only at positions $s\in S_{j'}$. Since $S_j$ and $S_{j'}$ are distinct, $x'$ and $y'$ are distinct, and $x$ and $y$ are distinct. Therefore, $G$ has edge dimension at most $k$ and clique number at least $\ek(k)$, so $\mc(k)\ge \ek(k)$.
    

    For the other direction, let $G$ be an arbitrary graph of edge metric dimension at most $k$ and clique number $\mc(k)$, so $G$ has $\mc(k)$ vertices that form a clique and $k$ vertices $u_1, \dots, u_k$ that form an edge resolving set. Note that some of the vertices in the edge resolving set may be in the clique. Consider the distance vectors of the edges in the clique with respect to the vertices $u_1, \dots, u_k$. There are at most $2$ possible values for each coordinate among the edges in the clique, since all vertices in the clique are neighbors. For each vertex $u_i$, let $p_i$ be the least distance of any vertex in the clique to $u_i$, so the only possible values of the coordinate corresponding to $u_i$ among the edges in the clique are $p_i$ and $p_i + 1$. 

    From $G$, we construct a family $\mathcal{F}$ of subsets of $\left\{1, \dots, k\right\}$ for which all pairwise unions are distinct. In particular, we create a subset $S_j$ of $\left\{1,\dots,k\right\}$ for each vertex $v_j$ in the clique of size $\mc(k)$. If $v_j$ has distance $p_i$ to $u_i$, then we include $i$ in $S_j$. Otherwise, if $v_j$ has distance $p_{i}+1$ to $u_i$, then we do not include $i$ in $S_j$. Thus, any edge $\left\{v_x, v_y\right\}$ in the clique has distance $p_i$ to $u_i$ if $i \in S_x \cup S_y$, and otherwise $\left\{v_x, v_y\right\}$ has distance $p_i + 1$ to $u_i$. Since all edges in the clique have unique distance vectors, all pairwise unions in $\mathcal{F}$ must be distinct. Thus, $\ek(k) \ge \mc(k)$, so $\mc(k) = \ek(k)$.
\end{proof}

\begin{cor}
\label{cor:ek-bound}
    We have $\ek(k) = O(2^{k/2})$ and $\ek(k) = \Omega((\frac{8}{3})^{k/6})$.
\end{cor}

In the next corollary, we bound the size of the largest element of a maximum family $\mathcal{F}$ of subsets of $\{1,2,\dots,k\}$ whose pairwise unions are distinct.

\begin{cor}
   Let $S$ be a largest-sized element of a maximum family $\mathcal{F}$ of subsets of $\{1,2,\dots,k\}$ whose pairwise unions are distinct. Both $|S|$ and $|\{1,2,\dots,k\}-S|$ are $\Theta(k)$.
\end{cor}
\begin{proof}
    Suppose that $S$ contains $m$ elements, then $\mathcal{F}$ contains at most one other element $S'$ that is a subset of $S$, and every other element of $\mathcal{F}$ contains an element of $\{1,2,\dots,k\}-S$. Moreover, for every element $S''$ of $\mathcal{F}$ that is not $S$ nor $S'$ the intersection $S''\cap \left(\{1,2,\dots,k\}-S\right)$ is unique. We have
    $$
    |\mathcal{F}|\le 2+\binom{k-m}{1}+\binom{k-m}{2}+\dots+\binom{k-m}{\min\{m,k-m\}}=1+\sum_{i=0}^{\min\{m,k-m\}}\binom{k-m}{i}.
    $$

    If $k-m=o(k)$, then $|\mathcal{F}|\le1+2^{k-m}$ contradicts the fact that $\ek(k)=\Omega\left((\frac{8}{3})^{k/6}\right)$
from Corollary~\ref{cor:ek-bound}. If $m=o(k)$, then
$$
|\mathcal{F}|\le1+\sum_{i=0}^{\min\{m,k-m\}}\binom{k-m}{i}\le (m+1)\binom{k}{m}\le(m+1)\frac{k^m}{m!}
$$
and
$$
\log|\mathcal{F}|=O\left(\log(m+1)+m\log k-m\log m\right)=O\left(m\log\frac{k}{m}\right).
$$
Since $k/m=\omega(1)$, we have $\log|\mathcal{F}|=O(m\log(k/m))=o(mk/m)=o(k)$.
This contradicts the fact that $\ek(k)=\Omega\left((\frac{8}{3})^{k/6}\right)$.
\end{proof}

\section{Conclusion}\label{sec:conclusion}

In this paper, we proved a number of bounds on the maximum possible price of fault tolerance for metric dimension, edge metric dimension, and truncated metric dimension. These bounds were sharp up to the base of the exponent, but it remains to sharpen them further. 

We saw that $\ftxdim(G) \ge \xdim(G)+1$ for all graphs $G$ and variants $\xdim(G)$ of metric dimension. For arbitrary variants $\xdim(G)$ of metric dimension, what else can be said in general about $\ftxdim(G)$ with respect to $\xdim(G)$? The same question can be asked for special families of graphs $G$ like trees, bipartite graphs, $k$-regular graphs, and $k$-degenerate graphs, as well as for variants $\xdim(G)$ of metric dimension with certain properties. 

Based on the observations about fault tolerance and maximum degree in this paper, we have the following question. For what variants $\xdim(G)$ of metric dimension is it true that \[\lim_{k \rightarrow \infty} \left( \max_{G: \text{ } \xdim(G) = k} \frac{\log(\ftxdim(G))}{k} \right) \neq \lim_{k \rightarrow \infty} \left( \max_{G: \text{ } \xdim(G) = k} \frac{\log(\Delta(G))}{k} \right)?\] 

In order to prove our bounds for fault tolerance, we also proved a number of results about the maximum possible degree of vertices in graphs of a given metric dimension, edge metric dimension, and truncated metric dimension. We note that there is a related open problem \cite{gkl} which remains unsolved. Specifically, what is the maximum possible minimum degree of a graph of metric dimension $k$?

An upper bound of $3^{k-1}$ was obtained in \cite{gkl}. We conjecture that this is sharp, but matching lower bound constructions have only been found \cite{gkl} for $k \le 3$. To our knowledge, the analogous problem for edge metric dimension and truncated metric dimension is also unsolved.

In Section~\ref{sec:dim_extremal}, we proved results about the maximum possible number of vertices within a distance of $j$ of a given vertex in a graph of metric dimension $k$. With the same proofs, we obtain the following results for truncated metric dimension.

\begin{thm}
    For any vertex $v$ in a graph of $k$-truncated metric dimension $j$, the maximum possible number of vertices $u \neq v$ with $\dist(u, v) \le i$ is equal to $(2i+1)^{j}-1$ for all $i$ with $k \ge 3i-1$.
\end{thm}

\begin{thm}
    For any vertex $v$ in a $k$-resolving set of size $j$, the maximum possible number of vertices $u \neq v$ with $\dist(u, v) = i \ge 1$ is equal to $(2i+1)^{j-1}$ for all $i$ with $k \ge 3i-1$. 
\end{thm}

\begin{cor}
    For any vertex $v$ in a $k$-resolving set of size $j$, the maximum possible number of vertices $u \neq v$ with $\dist(u, v) \le i$ is equal to \[\sum_{r = 1}^i (2r+1)^{j-1} = \frac{2^{j-1}}{j}i^k (1+o(1))\] for all $i$ with $k \ge 3i-1$. 
\end{cor}

It remains to investigate the case when $k < 3i-1$.

\end{document}